\documentclass{article}
\usepackage[utf8]{inputenc}
\DeclareUnicodeCharacter{2010}{-}

\usepackage[dvipsnames]{xcolor}

\usepackage{a4wide,amsmath,amssymb,amsthm}
\usepackage{fullpage}
\usepackage{enumerate}
\usepackage{latexsym,epsfig,graphicx}
\usepackage{apacite}
\usepackage[ruled, commentsnumbered, linesnumbered, vlined]{algorithm2e}
\usepackage{pdflscape}
\usepackage{float}\usepackage{float}

\newtheorem{theorem}{Theorem}
\newtheorem{lemma}[theorem]{Lemma}
\newtheorem{proposition}[theorem]{Proposition}

\newtheorem{corollary}[theorem]{Corollary}

\newcommand{\rafaelC}[1]{\textcolor{black}{#1}}

\newcommand{\marcioC}[1]{\textcolor{black}{#1}}

\title{Extended formulation and valid inequalities for the multi-item inventory lot-sizing problem with supplier selection}

\author{ 
    Leopoldo E. Cárdenas-Barrón {\thanks{Tecnologico de Monterrey, School of Engineering and Sciences,  Ave. Eugenio Garza Sada 2501, Monterrey, N.L., Mexico, 64849. ({\tt lecarden@tec.mx})}}
    \and Rafael A. Melo {\thanks{Universidade Federal da Bahia, Departamento de Ci\^{e}ncia da Computa\c{c}\~{a}o, Computational Intelligence and Optimization Research Lab (CInO), Salvador, Brazil.  ({\tt melo@dcc.ufba.br})}} 
	\and Marcio C. Santos {\thanks{Universidade fereral do Ceará, Campus Russas, Russas, Brazil.   ({\tt marciocs@ufc.br})}}
  }

\begin{document}

\maketitle

\begin{abstract}
    This paper considers the multi-item inventory lot-sizing problem with supplier selection. The problem consists \rafaelC{of} determining an optimal purchasing plan in order to satisfy dynamic deterministic demands for multiple items over a finite planning horizon, taking into account the fact that multiple suppliers are available to purchase from.
    As the complexity of the problem was an open question, we show that it is NP-hard.
    We propose a facility location extended formulation for the problem which can be preprocessed based on the cost structure and describe \rafaelC{new valid inequalities in the original space of variables}. Furthermore, we study the projection of the extended formulation into the original space and show the connection between the inequalities generated by this projection and the \rafaelC{newly proposed inequalities}. Additionally, we present a simple and easy to implement yet very effective MIP (mixed integer programming) heuristic using the extended formulation.
    \rafaelC{Besides, we introduce two new benchmark sets of instances to assess the performance of the approaches under different cost structures.}
    \rafaelC{Computational results show that the preprocessing approach can significantly reduce the size of the formulation to be solved, allowing both an increase in the number of instances solved to optimality within the time limit and a reduction on the average time to solve them.
    Moreover, the described inequalities can improve the performance of a standard formulation for nearly all instance groups. They can also be used to provide strong lower bounds for certain large instances for which the preprocessed facility location formulation fails even to provide a linear relaxation bound due to memory limitations.
    Furthermore, the proposed MIP heuristic outperforms the heuristics available in the literature as it obtains solution values which at least match those reported for all instance groups, strictly improving most of them.
    The results also show that the performance of the approaches can vary considerably under different cost structures.
    }
    \newline 

\end{abstract}

\section{Introduction}

In \rafaelC{the} nowadays competitive business environment, it has become more important to \rafaelC{achieve excellence in} supplier selection and lot-sizing processes for purchasing the products required by the companies.
\rafaelC{As suppliers may have varying availability for the needed products at different costs, the} intention of these processes is to choose the best suppliers from which to purchase the items, the amount of the lots, and the time to set the orders in a finite planning horizon. 
\rafaelC{
A company may have the single objective of minimizing the total procurement cost or, in certain situations, multiple objectives which may include costs, quality of items, among other metrics.
The multi-item inventory lot-sizing problem with supplier selection (MIILSPSS) is a problem belonging to this context.} 
\rafaelC{In the MIILSPSS, there exists a known dynamic deterministic demand for multiple items in a finite planning horizon, which can be purchased from a set of suppliers in each of the periods. There is a fixed supplier ordering cost that is charged for each period an order is put to that supplier, as well as unitary purchasing and holding costs. There are no capacities on the amounts of purchased items. The problem consists of determining a purchasing plan minimizing the total cost.} These sorts of problems regarding supplier selection and lot-sizing are faced by a variety of industries since companies that manufacture or simply distribute products very often need to make these decisions in an optimized fashion.   
\rafaelC{Note that uncapacitated problems occur when items are acquired from very large suppliers and, besides, can arise in solving approaches, such as column generation or Lagrangian relaxation, for even more challenging capacitated problems.}


\citeA{UstDem08} propose an integration of analytic network process (ANP) and achievement scalarizing functions for a multi-objective problem of choosing suppliers and defining purchasing quantities for a single item, taking into consideration tangible-intangible criteria.
\marcioC{\citeA{UstDem08b} integrate ANP and multi-objective mixed integer linear programming (MOMILP) for single-item lot-sizing problems with supplier selection.}
\citeA{ZhaKla12} consider a single-item lot-sizing problem with simultaneous supplier selection and provide a mixed integer programming (MIP) formulation for the problem together with a study of its underlying polytope. They provide necessary and sufficient conditions to obtain facet defining inequalities for the uncapacitated case and valid inequalities for the capacitated one.
\citeA{ChoSha13} propose a MIP formulation for a single-item lot-sizing and distribution problem with supplier selection. 
\citeA{ChoSha14} extend the problem studied in \citeA{ChoSha13} to a multi-objective setting and propose a goal programming approach.
\marcioC{\citeA{ArsRicGua16} study the polyhedron associated with a formulation for a two-echelon single-item lot-sizing problem with supplier selection which allows inventory at the suppliers. The authors show that the problem is NP-hard, provide a class of facet defining inequalities, and give the convex hull of the proposed formulation for certain special cases.} \citeA{GhaMaz17} consider a single-item lot-sizing problem with supplier selection, backlogging, and quantity discounts. The authors present a MIP formulation together with a recursive approach that can be used to solve the problem \rafaelC{iteratively}.
\citeA{AkbRap18} consider a single-item uncapacitated lot sizing problem with multi-mode replenishment and batch deliveries. The authors present an NP-hardness proof and show that the problem remains NP-hard even for very simple and strict cost structures. Additionally, they present a 2-approximation algorithm and show that the problem admits a fully polynomial-time approximation scheme (FPTAS).
\rafaelC{We remark that lot-sizing problems with subcontracting~\cite{AtaHoc01} are somehow related to lot-sizing problems with supplier selection, although subcontracting is more commonly used for capacitated problems whenever there is a lack of production capacity.}


\citeA{KasLee96} propose a MIP formulation for a multi-item supplier selection problem with lead times.
\citeA{Dah03} considers a multi-objective supplier selection and order quantities for a multi-item problem with quantity discounts and propose a preference-based approach.
\citeA{BasLeu05} studied the multi-item inventory lot-sizing problem with supplier selection (MIILSPSS), which is the problem considered in our work.
The authors proposed an exhaustive enumerative search and a heuristic based on the Wagner-Whitin algorithm which consists of \rafaelC{construction and improvement phases}.
\rafaelC{\citeA{WadRav07} compare the use of three multi-objective approaches for a multi-objective supplier selection problem for multiple items considering objectives as price, lead-time, and quality.}
\marcioC{
\citeA{Rezaei08.2} propose a genetic algorithm for a multi-item inventory lot-sizing problem with supplier selection in which the items provided by the supplier may have imperfections.}
\marcioC{\citeA{SadAfsSoh08} propose a hybrid intelligent algorithm which combines a fuzzy neural network with a genetic algorithm to plan and control the inventory at different levels depending on demand rates for a multi-item inventory lot-sizing problem with supplier selection.}
\marcioC{\citeA{woarawichai11} propose a MIP formulation for a multi-item inventory lot-sizing problem with supplier selection under storage and budget constraints.}
\marcioC{\citeA{Rezaei11} consider capacitated multi-objective multi-item lot-sizing problems with multiple suppliers. The problems are modeled as multi-objective mixed integer nonlinear programs (MOMINLP) and tackled with multi-objective genetic algorithms to attempt encountering Pareto-optimal solutions. The authors observed that the problems are NP-hard as they extend the capacitated lot-sizing problem~\cite{BitYan82}.}
\citeA{WarSinBan14} formulate and solve a mixed integer nonlinear programming model to manage \rafaelC{a dynamic supplier selection problem which takes into consideration supplier capacity, product quality, and lead times.}
\citeA{CarGonTre15} have revisited and tackled the MIILSPSS~\cite{BasLeu05} with a heuristic based on the reduce and optimize approach (ROA). The authors have demonstrated through computational experiments that their heuristic obtains better solutions when compared with the methods of \citeA{BasLeu05}.
\marcioC{
\citeA{Alfares18} propose a MIP formulation and heuristic algorithms for the multi-item inventory lot-sizing problem with supplier selection with the possibility of shortages.}
\citeA{CunSanMorBar18} consider the integration of a multi-item lot-sizing problem with supplier selection for raw material purchasing in the chemical industry. The authors propose a MIP formulation \rafaelC{that} is solved using a commercial solver and compare the advantages of such \rafaelC{an} integrated approach over a non-integrated one in which the two problems are solved sequentially.
\marcioC{\citeA{KirMei19} consider the application of a lot-sizing problem with supplier selection in the process industry and propose a kernel search heuristic to tackle the real instances of their case study.}

\rafaelC{A summary of the main characteristics and solution approaches for related problems is provided in Table~\ref{summary:articles}.}
The interested reader is referred to \citeA{AisHaoHas07}, \citeA{HoXuDey10} and \citeA{WarSinBan12} for surveys regarding different aspects of supplier selection and order lot-sizing.

\rafaelC{
Extended formulations such as facility location~\cite{KraBil77}, multi-commodity~\cite{RarCho79} and shortest path~\cite{EppMar87} have been extensively used to tackle lot-sizing problems~\cite{SolSur11,AkaMil12,CarNas16,CunMel16b,MelRib17}. Such formulations, however, can become computationally intractable as the problem size increases. Several authors have analyzed the projections of these extended formulations for lot-sizing problems.
\citeA{RarWol93} analyzed the projection of the multicommodity formulation for uncapacitated fixed charge network problems into the original space of variables and showed that such projection is given by the so-called dicut inequalities. \citeA{AkaMil09} and \citeA{AkaMil12} discuss the relationship between valid inequalities and extended formulations for multi-level lot-sizing problems.
\citeA{ZhaKucYam12} performed a polyhedral study for a multi-level lot-sizing with intermediate demands and showed that their described inequalities give the nondominated inequalities implied by the projection of the multicommodity formulation for the two-level problem.
\citeA{MelRib17} performed theoretical and computational comparisons of reformulations and valid inequalities for a multi-item uncapacitated lot-sizing with inventory bounds.
In an attempt to overcome the issue regarding the prohibitive size of extended formulations for large problems, \citeA{VanWol06} introduced the concept of approximate extended formulations, in which a parameter is used to control their sizes. 
}



\subsection{\rafaelC{Main contributions and organization}}

\rafaelC{
The main contributions of our work can be summarized as follows. We firstly show that the multi-item inventory lot-sizing problem with supplier selection is NP-hard.
Secondly, we propose a facility location extended formulation together with an effective preprocessing scheme~\cite{EenBie05,Sav94} and new valid inequalities in the original space of variables.
Differently from approximate extended formulations~\cite{VanWol06}, our preprocessing approach attempts to reduce the size of the formulation while guaranteeing the same quality in the provided lower bounds.
Thirdly, we consider the projection of the facility location extended formulation into the original space. Fourthly, given that MIP heuristics have been successfully applied for several production planning and lot-sizing problems \cite{AkaMil09,HelSah10,MelWol12,MelRib17,CunKraMel19}, we propose a simple and easy to implement yet very effective MIP heuristic.
Last but not least, we introduce two new sets of benchmark instances to analyze the performance of the proposed approaches under alternative cost structures. 
}

\begin{landscape}
\begin{table}[H]

\small
\color{black}
    \centering
    \begin{tabular}{l| p{0.8cm}|p{2.2cm}|p{6.3cm}|p{6.3cm}} \hline
         & Multi-item & Mono or multi-objective & Other characteristics & Solution approach \\ \hline
         \citeA{KasLee96} & Yes & Mono-objective & Stochastic demands and items lost & Chance-constrained integer programming \\[1.5pt]
         \citeA{Dah03} & Yes & Multi-objective & Items can be rejected, missing items and price discounts & MIP formulation with preference oriented approach \\[1.5pt]
         \citeA{BasLeu05} & Yes &  Mono-objective & - & Enumerative algorithm and heuristics \\[1.5pt]
         \citeA{WadRav07} & Yes &  Multi-objective & Capacity constraints & Goal programming \\[1.5pt]
         \citeA{Rezaei08.2} & Yes &  Multi-objective & Transportation costs and product quality levels & Genetic algorithms \\[1.5pt]
         \citeA{SadAfsSoh08} & No &  Mono-objective & Real case study with uncertain demands & Data analysis using neural networks for prediction and genetic algorithms \\[1.5pt]
         \citeA{UstDem08b} & No &  Multi-objective & Defect rate analysis and defective products & Analytic network process integrated with MIP formulation \\[1.5pt]
         \citeA{UstDem08} & No &  Multi-objective & Defective products & Goal programming\\[1.5pt]
         \citeA{Rezaei11} & Yes &  Multi-objective & Back-order and supplier quality level & Nonlinear MIP formulation \\[1.5pt]
         \citeA{woarawichai11} & Yes &  Mono-objective & Storage capacity and purchase budget & MIP formulation \\[1.5pt]
         \citeA{ZhaKla12} & No &  Mono-objective & \marcioC{Uncapacitated} and capacitated cases & Theoretical polyhedral study \\[1.5pt]
         \citeA{ChoSha13} & No &  Mono-objective & Service levels requirements and cost of rejected items & MIP formulation \\[1.5pt]
         \citeA{ChoSha14} & No &  Multi-objective & Storage space constrains and carrier selection & Preemptive goal programming \\[1.5pt]
         \citeA{WarSinBan14} & Yes &  Mono-objective & Supplier capacity, product quality, and lead times  & Nonlinear MIP formulation\\[1.5pt]
         \citeA{CarGonTre15} & Yes &  Mono-objective & - &  Reduce and optimize approach (ROA)  \\[1.5pt]
         \citeA{ArsRicGua16} & No & Mono-objective & Two-echelon, suppliers can hold inventory, supplier controls shipments & Dynamic programming, branch and cut algorithm, and polyhedral study\\[1.5pt]
         \citeA{GhaMaz17} & No &  Mono-objective & Backlogging, quantity discounts & MIP formulation and forward dynamic programming models \\[1.5pt]
         \citeA{Alfares18} & Yes &  Mono-objective & Backlogging, quantity discounts, lead times & MIP formulation, modified Silver-Meal heuristic and genetic algorithm \\[1.5pt]
         \citeA{AkbRap18} & No &  Mono-objective & Batch deliveries and stepwise replenishment costs & Constructive 2-approximation polynomial algorithm and fully polynomial-time approximation scheme (FPTAS) \\[1.5pt]
         \citeA{CunSanMorBar18} & No &  Mono-objective & Integrated raw material purchasing and production planning, production and inventory capacities, lead times, discount levels & MIP formulation \\[1.5pt]
         \citeA{KirMei19} & No &  Mono-objective & Storage selection and discounts & MIP formulation and adapted kernel heuristic \\
         \hline
    \end{tabular}
    \caption{Summary of the literature on lot-sizing with supplier selection.}
    \label{summary:articles}
\end{table}
\end{landscape}

The remainder of this paper is organized as follows. Section~\ref{sec:problemdefinition} formally defines the multi-item inventory lot-sizing problem with supplier selection and shows that the problem is NP-hard.
Section~\ref{sec:proposedapproaches} presents the facility location extended formulation together with the preprocessing scheme and describes the proposed $(l,S_j)$-inequalities. 
Section~\ref{sec:projection} analyzes the projection of the extended formulation into the original space and shows how it relates to the $(l,S_j)$-inequalities. 
Section~\ref{sec:mipheuristic} details the proposed MIP heuristic.
Section~\ref{sec:computationalexperiments} summarizes the performed computational experiments.
Section~\ref{sec:conclusions} discusses final comments.

\section{Problem definition and standard mixed integer programming formulation}
\label{sec:problemdefinition}

In this section, we formally introduce the multi-item inventory lot-sizing problem with supplier selection (MIILSPSS) and describe a standard mixed integer programming formulation for the problem. After that, in Subsection~\ref{sec:nphardness}, we show that the problem is NP-hard.

The MIILSPSS can be formally defined as follows. Consider $I=\{1,\ldots,NI\}$ to be the set of items, $J=\{1,\ldots,NJ\}$ to be the set of suppliers and $T=\{1,\ldots,NT\}$ to be the set of periods composing the planning horizon. A deterministic \rafaelC{time-varying} demand $d^i_t \geq 0$ must be met without backlogging for each item $i\in I$ in each period $t\in T$.
There is a unitary purchasing price $P_{ij}$ of item $i \in I$ from supplier $j\in J$. A transaction cost $O_j$ for supplier $j\in J$ is incurred whenever \rafaelC{an} item is purchased from $j$ in a given period. Furthermore, a \rafaelC{per-unit} holding cost $H_i$ is incurred for item $i\in I$ in every period the item is held in stock. The problem consists \rafaelC{of} determining a purchasing plan which minimizes the total cost. Let $d^i_{kt}=\sum_{l=k}^t d^i_l$ be the \rafaelC{cumulative} demand for item $i\in I$ in periods from $k\in T$ up to $t\in T$, with $k\leq t$. It is assumed that all the costs are nonnegative and that there are no initial or final stocks.  \rafaelC{Table~\ref{tab:summarynotation} summarizes the main notation used throughout the paper.}

Consider variable $x^{ij}_t$ to be the amount of item $i \in I$ purchased from supplier $j \in J$ in period $t \in T$, and variable $y^j_t$ to be equal to one if items are purchased from supplier $j \in J$ in period $t \in T$ and to be equal to zero otherwise. The problem can thus be formulated as \cite{BasLeu05}:

\begin{align}   z_{STD} =& \min \ \sum_{i=1}^{NI} \sum_{j=1}^{NJ} \sum_{t=1}^{NT} P_{ij} x^{ij}_{t}  + \sum_{j=1}^{NJ} \sum_{t=1}^{NT} O_j y^j_t +  \sum_{i=1}^{NI} \sum_{t=1}^{NT} H_i \left( \sum_{j=1}^{NJ} \sum_{k=1}^t x^{ij}_k - d^i_{1t} \right)  \label{STDobj} \\
& \sum_{j=1}^{NJ} \sum_{k=1}^t x^{ij}_k \geq d^i_{1t}, \qquad \textrm{for } i \in I, \ t \in T, \label{STD1}\\
& x^{ij}_t  \leq M y^j_t, \qquad \textrm{for }  i \in I,\ j\in J,\ t \in T, \label{STD2}\\
& y^j_{t} \in \{0,1\}, \qquad \textrm{for } j \in J, \ t \in T, \label{STD3}\\
& x^{ij}_{t} \geq 0, \qquad \textrm{for } i \in I,\ j \in J, \ t \in T.  \label{STD4}
\end{align}
The objective function (\ref{STDobj}) minimizes the total sum of purchasing, transaction, and storage costs. 
Constraints (\ref{STD1}) guarantee that all the demands are satisfied. Constraints (\ref{STD2}) ensure the setup variables are set to one whenever items are purchased from a supplier in a given period. Constraints (\ref{STD3}) and (\ref{STD4}) impose, respectively, the integrality and nonnegativity requirements on the variables. This formulation has $O(NI \times NJ \times NT)$ variables and constraints. \rafaelC{Note that even though the studied problem considers \rafaelC{time-independent} costs, this standard formulation as well as the results proposed in the remainder of the paper can be easily extended for the variant of the problem with \rafaelC{time-dependent} costs.}

\begin{table}[H]
    \centering
    \small
    \color{black}
    \begin{tabular}{cp{11 cm}} 
        \hline
	\textbf{Notation} & \textbf{Description} \\
	\hline
 \rule{0pt}{3ex}        $I$ & Set of items \\
        $NI$ & Number of items, i.e., $|I|$ \\
        $J$ & Set of suppliers\\
        $NJ$ & Number of suppliers, i.e., $|J|$\\
        $T$ & Planning horizon, given by a set of periods \\
        $NT$ & Size of the planning horizon, i.e., $|T|$\\
        $d_{t}^{i}$ & Demand for item $i$ in time $t$\\
        $P_{ij}$ & Per unit price of purchasing item $i$ from supplier $j$\\
        $O_{j}$ & Transaction cost for supplier $j$ \\
        $H_i$ & Per unit holding cost of item $i$  \\
        $d_{kt}^{i}$ & Cumulative demand for item $i$ in the interval of periods $[k,t]$, i.e., $\sum_{k'=k}^t d_{k'}^{i}$  \\ \hline
 \rule{0pt}{3ex}       $F$ & Potential facility locations \\
        $NF$ & Number of potential facility locations, i.e., $|F|$\\
        $C$ & Set of clients \\
        $NC$ & Number of clients, i.e., $|C|$\\
        $q_f$ & Fixed cost to open facility $f$ \\
        $v_{cf}$ & Cost of serving client $c$ from facility $f$ \\ \hline
 \rule{0pt}{3ex}       $x_{t}^{ij}$ & Continuous variable representing the amount of item $i$ purchased from supplier $j$ in period $t$ \\
        $y^{j}_{t}$ & Binary variable representing whether or not items are purchased from supplier $j$ in period $t$ \\
        $X_{tk}^{ij}$ & Facility location variable representing the amount of item $i$ purchased from supplier $j$ in period $t$ to satisfy demand of period $k$ \\ \hline
 \rule{0pt}{3ex}       $\theta_{t}^{ij}$ & Dual variable associated with constraints \eqref{FLS1} \\
        $\gamma_{tk}^{ij}$ & Dual variable associated with constraints \eqref{FLS2}  \\
        $\phi_{t}^{i}$ & Dual variable associated with constraints \eqref{FLS3}\\
        \hline
    \end{tabular}
    \caption{Summary of the used notation.}
    \label{tab:summarynotation}
\end{table}

\subsection{NP-hardness}
\label{sec:nphardness}

\rafaelC{Certain uncapacitated production planning problems are known to be solved in polynomial time~\cite{WagWhi58,Zan69,MelWol10,ZhaKla12}, while others were shown to be NP-hard. Examples of NP-hard problems include the joint-replenishment problem, the one-warehouse multi-retailer problem~\cite{ArkJonRou89, CunMel16}, and variants of the uncapacitated multi-plant lot-sizing with inter-plant transfers~\cite{CunMelKra20}. In this regard, to} the best of our knowledge, there is no available NP-hardness proof for the multi-item inventory lot-sizing problem with supplier selection. In this section, we show that the problem is NP-hard via a reduction from the \rafaelC{NP-hard} uncapacitated facility location problem\rafaelC{~\cite{CorNemWol90}}. 

The uncapacitated facility location problem (UFL) can be formally defined as follows. Consider a set $F=\{1,\ldots,NF\}$ of potential facility locations, a set $C=\{1,\ldots,NC\}$ of clients, a fixed cost $q_f$ to open facility $f\in F$, and a cost $v_{cf}$ of serving client $c \in C$ from facility $f \in F$. The problem consists \rafaelC{of} obtaining a subset $F'\subseteq F$ of the facilities to be opened and then to assign clients to these facilities while minimizing the total cost. The decision version of the problem asks whether there is a solution with \rafaelC{a} cost less than or equal to a value $K$ which is given as input.

\begin{theorem}
The multi-item inventory lot-sizing problem with supplier selection is NP-hard.
\end{theorem}

\begin{proof}
\rafaelC{In what follows, we present a polynomial transformation UFL$\propto$MIILSPSS. Namely, we} show how an instance for the decision version of the MIILSPSS is obtained from an instance of the decision version of the UFL. \rafaelC{As the decision version of UFL is NP-complete, this implies that the decision version of MIILSPSS is also NP-complete and, thus, its optimization version is NP-hard. Given an instance for UFL, a corresponding instance for MIILSPSS can be obtained as follows.} 
Create a supplier for each potential facility $f \in F$, an item for each client $c\in C$, and set the number of periods as $NT = 1$. For each supplier $f\in J$, set its transaction cost as the cost of opening the corresponding facility, i.e., $O_f = q_f$. The cost of acquiring item $c \in I$ from supplier $f\in J$ is set as the cost of serving client $c\in C$ from facility $f\in F$, i.e., $ P_{cf} = v_{cf}$. 
We now show that the instance for UFL has a solution with value less than or equal to $K$ if and only if the corresponding instance for MIILSPSS has a solution with value less than or equal to $K$. 
Consider a solution $F' \subseteq F$ for the uncapacitated facility location in which each facility $f\in F'$ serves a set $C_f$ of clients, with cost $K=\sum_{f \in F'} q_f + \sum_{f \in F'}\sum_{c \in C_f} v_{cf}$. Thus, there is a solution for the multi-item inventory lot-sizing problem with supplier selection with only nonzero values $y^{f}_1=1$ for $f\in F'$ and $x^{cj}_1=1$ for $j\in F'$ and $c \in C_j$, whose objective is $\sum_{f \in F'}\sum_{c \in C_f} v_{cf} x^{cf}_1 + \sum_{f \in F'} q_f y^f_1 = K$ . 
Now consider a solution $(\hat{x},\hat{y})$ for the multi-item lot-sizing problem with supplier selection with cost $K=\sum_{i \in NI}\sum_{j \in NJ} P_{ij} x^{ij}_1 + \sum_{j \in NJ} O_{j} y^{j}_1$, and assume that $\hat{x}$ is integral (note that such integral solution always exist when we consider solutions with the lowest possible cost). Observe that there is a corresponding solution $F' = \{f \in J \ | \ y^f_1 = 1 \}$ and $C_f = \{ c \in C \ | \ x^{cf}_1 = 1 \}$ for the uncapacitated facility location with cost $\sum_{j \in F'} O_{j} +  \sum_{j \in F'}\sum_{i \in C_j} P_{ij} = K$.
\end{proof}

\section{Extended formulation and valid inequalities}
\label{sec:proposedapproaches}

In this section we present the extended formulation and valid inequalities proposed in this paper. Subsection~\ref{sec:facilitylocation} describes the facility location extended formulation \rafaelC{and the preprocessing approach}. Subsection~\ref{sec:validinequalities} introduces the new $(l,S_j)$-inequalities.

\subsection{Facility location extended formulation}
\label{sec:facilitylocation}

Define variable $X^{ij}_{tk}$ to be the amount of item $i \in I$ purchased from supplier $j \in J$ in period $t\in T$ to satisfy demand of period $k \in T$, with $t\leq k$. A facility location formulation~\cite{KraBil77} can be cast as
\begin{align}   
z_{FL} = & \min \ \sum_{i=1}^{NI} \sum_{j=1}^{NJ} \sum_{t=1}^{NT} \sum_{k=t}^{NT} P_{ij} X^{ij}_{tk}  + \sum_{j=1}^{NJ} \sum_{t=1}^{NT} O_j y^j_t +  \sum_{i=1}^{NI} \sum_{t=1}^{NT} H_i (\sum_{u=1}^t \sum_{k=t+1}^{NT} X^{ij}_{uk})  \label{FLobj} \\
& \sum_{j=1}^{NJ}\sum_{t=1}^{k}X^{ij}_{tk} = d^i_k, \qquad \textrm{for } i \in I, \ k \in T, \label{FL1}\\
& X^{ij}_{tk}  \leq d^i_k y^j_t, \qquad \textrm{for }  i \in I,\ j\in J,\ t \in T,\ k \in \{t,\ldots,NT\}, \label{FL2}\\
& y^j_{t} \in \{0,1\}, \qquad \textrm{for } j \in J, \ t \in T, \label{FL3}\\
& X^{ij}_{tk} \geq 0, \qquad \textrm{for } i \in I,\ j \in J, \ t \in T, \ k\in \{t,\ldots,NT\}.  \label{FL4}
\end{align}
The objective function (\ref{FLobj}) minimizes the total sum of purchase, transaction, and storage costs. 
Constraints (\ref{FL1}) guarantee that all the demands are satisfied. Constraints (\ref{FL2}) enforce the  setup variables to one whenever items are purchased from a supplier in a given period. Constraints (\ref{FL3}) and (\ref{FL4}) are integrality and nonnegativity restrictions on the variables.

\subsubsection{Preprocessing the facility location extended formulation}\label{sec:preprocess}

The facility location extended formulation has a large number of variables since, in order to properly model the problem, it considers the possibility that the demand of a period is satisfied by the production in any other period sooner in the \rafaelC{planning} horizon. Although it might be the case for certain instances that the productions of the first periods are used to meet the demands \rafaelC{of} the later periods, it does not seem to be the case for real instances. In this context, we show that variables can be eliminated from the formulation based on the cost structure without losing optimality.

\begin{proposition}\label{prop:preprocessing}
    Let $X^{ij}_{tk}$ be a variable for which \marcioC{$O_j \leq (k-t) \times H_i \times d^i_k$}. Thus, we can set $X^{ij}_{tk'}=0$ for every $k'$ such that $k \leq k' \leq NT$ without losing optimality.
\end{proposition}
\begin{proof}
    Firstly, let $k$ be the earliest period after $t$ for which $O_j \leq \rafaelC{(k-t)} \times H_i \times d^i_k$ and assume there is an optimal solution in which $X^{ij}_{tk}=w>0$. Note that the right-hand side of the condition corresponds to the total storage cost for the demand of period $k$ which was purchased in period $t$. Thus, as the purchasing costs are \rafaelC{time-independent}, we can set $X^{ij}_{tk}=0$, $y^{j}_k=1$ and $X^{ij}_{kk}=w$ in order to obtain a solution which is at least as good as the previous one. Secondly, due to the property of extreme feasible solutions in fixed-charge networks, which was also observed in \citeA{BasLeu05}, there exists an optimal solution in which purchasing in a given period satisfies \rafaelC{the} demands of consecutive periods. Therefore, if there is an optimal solution in which $X^{ij}_{tk}=0$, there exists an optimal solution in which $X^{ij}_{tk'}=0$ for every $k < k' \leq NT$. 
\end{proof}

\rafaelC{Observe that the potential elimination of variables can strongly vary with input data and its strength relies on the relationship between transaction costs, holding costs, demands, and the length of the planning horizon. Besides, note} that Proposition~\ref{prop:preprocessing} can be easily adapted for the case in which costs are \rafaelC{time-dependent}.

\subsection{The $(l,S_j)$-inequalities}
\label{sec:validinequalities}

We describe the $(l,S_j)$-inequalities, \rafaelC{which are exponential in number, and} generalize the $(l,S)$-inequalities for the uncapacitated lot-sizing~\cite{BarVanWol84} by considering the different suppliers. \rafaelC{Besides, the inequalities we describe in the following are strongly related to those presented in \citeA{ZhaKla12}, but we remark that the uncapacitated single-item problem in their work is different from our problem with a unique item and the inequalities described therein are thus different from ours.} Define $L = \{1,\ldots,l\}$ with $1 \leq l \leq NT$, and $S_j \subseteq L$ for each $j\in J$.

\begin{theorem} The $(l,S_j)$ inequalities
\begin{equation}\label{ineq:lsj}
     \sum_{j=1}^{NJ} \left( \sum_{u\in L\setminus S_j } x^{ij}_u + \sum_{u\in S_j} y^{j}_u d^i_{ul} \right) \geq d^i_{1l}, \ \ \rafaelC{\textrm{for } i\in I,}
\end{equation}
are valid for the multi-item inventory lot-sizing with supplier selection.
\end{theorem}
\begin{proof}
    Let $(\hat{x},\hat{y})$ be a feasible solution for (\ref{STD1})-(\ref{STD4}). Firstly, consider the case in which $\hat{y}^j_k=0$ for every $j\in J$ and $k \in S_j$. This implies that $\hat{x}^{ij}_k=0$ for every $j\in J$ and $k \in S_j$, and thus constraints (\ref{STD1}) ensure that $\sum_{j=1}^{NJ} \sum_{u\in L\setminus S_j } \hat{x}^{ij}_u  \geq d^i_{1l}$. Now, assume that $\hat{y}^j_k=1$ for at least one $j\in J$ and $k \in S_j$, and let $k'$ be the earliest period in which this happens and $j'$ be the corresponding supplier. 
    As constraints (\ref{STD1}) ensure that $\sum_{j=1}^{NJ} \sum_{u\in \{1,\ldots,k'-1\} } \hat{x}^{ij}_u \geq d^i_{1,k'-1}$, the fact that $\hat{y}^{j'}_{k'}=1$ implies that $\sum_{j=1}^{NJ} \sum_{u\in \{1,\ldots,k'-1\} } \hat{x}^{ij}_u + \hat{y}^{j'}_{k'} d^i_{k',l}  \geq d^i_{1l}$. Therefore, the inequality is valid. 
\end{proof}


\subsubsection{Separation of the $(l,S_j)$-inequalities}
\label{sec:separation}

\rafaelC{As there is an exponential number of $(l,S_j)$-inequalities, we consider the separation problem for such inequalities. Given a} fractional solution $(\bar{x},\bar{y})$ for the linear relaxation of (\ref{STD1})-(\ref{STD4}) we want to find a most violated $(l,S_j)$-inequality (\ref{ineq:lsj}). 

These inequalities can be separated similarly to the $(l,S)$-inequalities for the uncapacitated lot-sizing. For every possible item $i \in I$ and time period $l \in T$ we can find a most violated $(l,S_j)$-inequality (\ref{ineq:lsj}) \rafaelC{using inspection} by simply computing $\sum_{j=1}^{NJ}\sum_{k=1}^l \min \{ \bar{x}^{ij}_k , d^{i}_{kl} \bar{y}^{j}_{k} \}$ and building $S_{j}$ appropriately according to the choices on the inner minimum in $O(NJ\times NT)$. This gives an $O(NJ\times NT^2)$ algorithm to separate the inequalities for each $i\in I$. In what follows, we present an $O(NJ\times NT\times \log NT)$ dynamic programming separation algorithm to encounter a most violated $(l,S_j)$-inequality for each $i\in I$.

Given an item $i \in I$, define $\alpha^i_l = \sum_{j=1}^{NJ}\sum_{k=1}^l \min \{ \bar{x}^{ij}_k , d^{i}_{kl} \bar{y}^{j}_{k} \}$. Inequality (\ref{ineq:lsj}) is violated for $L=\{1,\ldots, l\}$ whenever $\alpha^i_l < d^i_{1l}$. Note that the nonnegativity of the demands implies $d^i_{kl} \bar{y}^{j}_k \leq d^i_{ku} \bar{y}^{j}_k$  for $k \leq l < u$. For $j\in J$ and $k\in T$, define $l^j_k$ as the first period in which $d^i_{k,l^j_k-1}\bar{y}^{j}_k < \bar{x}^{ij}_k \leq d^i_{kl^j_k}\bar{y}^{j}_k$ and let $Y^j_l = \{k \in L\ | \ l^j_k > l \}$ and $\rafaelC{Z}^j_l = \{ j\in L \ | \ l^j_k = l \}$. Therefore, the value $\alpha_l$ can be determined using the recursion
\begin{equation}\label{eq:alphas}
\alpha^i_l = \alpha^i_{l-1} + d^i_l (\sum_{j\in J}\sum_{k\in Y^j_l} \bar{y}^{j}_k) + \sum_{j\in J}\sum_{k \in \rafaelC{Z}^j_l}(\bar{x}^{ij} - d^i_{k,l-1}\bar{y}^{j}_k),
\end{equation}
with $\alpha^i_0 = 0$ as base case.

Considering the fact that $Y^j_l = Y^j_{l-1} \cup \{l\} \setminus \rafaelC{Z}^j_l$, each period $k$ enters at most once in $Y^j$, leaves $Y^j$ and enters $\rafaelC{Z}^j$ at most once. Thus, all the $\alpha_l$ values can be determined in $O(NJ \times NT)$. Observe also that we can determine $l^j_k$ for each $j\in J$ and $k \in T$ in $O(\log NT)$ using binary search, implying a running time of $O(NJ\times NT\times \log NT)$ for all the calculations.

\section{On the projection of the facility location formulation}
\label{sec:projection}

In this section, we study the projection of the facility location extended formulation (\ref{FL1})-(\ref{FL4}) into the space of the original $(x,y)$ variables. We consider the extended formulation as a separation problem in order to describe the inequalities generated by its projection. After that, we show how they relate with the $(l,S_j)$-inequalities, showing that the linear relaxation of the facility location extended formulation provides the same bound as that of the linear relaxation of the standard formulation together with the $(l,S_j)$-inequalities. 

Given a fractional solution $(\hat{x}, \hat{y})$ feasible for the linear relaxation of (\ref{STD1})-(\ref{STD4}), we wish to find an inequality implied by the facility location extended formulation in the original space cutting off this solution. Consider the formulation
\begin{align}   
z_{FLS} = & \max \ 0  \label{FLSobj} \\
& \sum_{j=1}^{NJ}\sum_{t=1}^{k}X^{ij}_{tk} = d^i_k, \qquad \textrm{for } i \in I, \ k \in T, \label{FLS1}\\
& X^{ij}_{tk}  \leq d^i_k \hat{y}^j_t, \qquad \textrm{for }  i \in I,\ j\in J,\ t \in T,\ k \in \{t,\ldots,NT\}, \label{FLS2}\\
& \sum_{k=t}^{NT} X^{ij}_{tk} \leq \hat{x}^{ij}_t, \qquad \textrm{for } i \in I,\ j \in J, \ t \in T, \label{FLS3}\\
& X^{ij}_{tk} \geq 0, \qquad \textrm{for } i \in I,\ j \in J, \ t \in T, \ k\in \{t,\ldots,NT\}.  \label{FLS4}
\end{align}
The objective function simply maximizes an arbitrary constant.
Constraints (\ref{FLS1}) ensure all the demands are satisfied.
Constraints (\ref{FLS2}) limit the multicommodity purchasing variables considering the values in $\hat{y}$.
Constraints (\ref{FLS3}) link the original facility location variables with the values assumed by the original $\hat{x}$. Note that due to the nonnegativity of all the coefficients in the objective function (\ref{STDobj}), $\sum_{j\in J}\sum_{t \in T} \hat{x}^{ij}_t = d^i_{1,NT}$ and thus (\ref{FLS3}) will hold at equality.  Constraints (\ref{FLS4}) are nonnegativity requirements on the variables.

Define $\phi$ , $\gamma$ and $\theta$ to be the dual variables associated to constraints (\ref{FLS1}), (\ref{FLS2}) and (\ref{FLS3}), respectively. The dual of (\ref{FLSobj})-(\ref{FLS4}) can thus be written as
\begin{align}   
z_{DFLS} = & \min \ \sum_{i=1}^{NI} \sum_{j=1}^{NJ} \sum_{t=1}^{NT} \theta^{ij}_t \hat{x}^{ij}_{t} + \sum_{i=1}^{NI} \sum_{j=1}^{NJ} \sum_{t=1}^{NT} \sum_{k=t}^{NT} \gamma^{ij}_{tk} d^{i}_k \hat{y}^j_{t} + \sum_{i=1}^{NI}\sum_{t=1}^{NT} \phi^i_t d^i_t  \label{FLSDobj} \\
& \theta^{ij}_t + \gamma^{ij}_{tk} + \phi^i_k \geq 0, \qquad \textrm{for } i \in I,\ j \in J, \ t \in T, \ k\in \{t,\ldots,NT\}, \label{FLSD1}\\
& \theta^{ij}_t \geq 0, \qquad \textrm{for } i \in I,\ j \in J, \ t \in T,  \label{FLSD2}\\
& \gamma^{ij}_{tk} \geq 0, \qquad \textrm{for } i \in I,\ j \in J, \ t \in T, \ k\in \{t,\ldots,NT\}.  \label{FLSD3}
\end{align}

Note that variables $\phi$ are the only negative ones in \rafaelC{an} extreme ray (\ref{FLSDobj}) with \rafaelC{a} negative cost. Thus, we normalize the extreme rays by assuming without loss of generality that $\phi^i_t \geq -1$ for $i\in I$ and $t\in T$. We formalize the inequalities obtained via (\ref{FLSDobj})-(\ref{FLSD3}) as
\begin{equation}\label{ineq:projected}
     \sum_{i=1}^{NI} \sum_{j=1}^{NJ} \sum_{t=1}^{NT} \theta^{ij}_t {x}^{ij}_{t} + \sum_{i=1}^{NI} \sum_{j=1}^{NJ} \sum_{t=1}^{NT} \sum_{k=t}^{NT} \gamma^{ij}_{tk} d^{i}_k {y}^j_{t} + \sum_{i=1}^{NI}\sum_{t=1}^{NT} \phi^i_t d^i_t \geq 0.
\end{equation}

In what follows, we want to show that the matrix associated with constraints (\ref{FLSD1}) is totally unimodular, and in order to do so, we use the next two \rafaelC{\textit{very well known}} results.

\begin{theorem}\label{theo:transpose}
A matrix $A$ is TU iff: (a) the transpose matrix $A^T$ is TU iff (b) the matrix $(A,I)$ is TU, where $I$ denotes the identity matrix. \cite{HofKru57}
\end{theorem}

\begin{theorem}\label{theo:partitionmatrix}
A matrix $A$ is TU if: (a) $a_{ij} \in \{-1,0,+1 \}$ for all $i,j$, and (b) for any subset $M$ of the rows, there exists a partition $(M_1,M_2)$ of $M$ such that each column $j$ satisfies $\left| \sum_{i\in M_1} a_{ij} - \sum_{i\in M_2} a_{ij} \right| \leq 1$. \citeA{Gho62}
\end{theorem}

\begin{theorem}\label{theo:matrixTU}
    The matrix associated with constraints (\ref{FLSD1}) is totally unimodular.
\end{theorem}
\begin{proof}
Denote $A$ the matrix associated to constraints (\ref{FLSD1}). Let $A=(B,I)$, where $B$ is the submatrix with the columns corresponding to variables $\theta$ and $\phi$ and $I$ is the identity submatrix with those columns related to the $\gamma$ variables. Using Theorem~\ref{theo:transpose}, we can concentrate on $B$, as $A=(B,I)$ is totally unimodular if $B$ is totally unimodular. Furthermore, we focus on $B^T$ and show that the properties in Theorem~\ref{theo:partitionmatrix} hold. Property (a) clearly holds. Now, given $M$ we add to $M_1$ the lines associated to the $\phi$ variables and to $M_2$ those associated to the $\theta$ variables. Thus the result holds.  
\end{proof}

We now analyze nondominated inequalities (\ref{ineq:projected}) obtained via (\ref{FLSDobj})-(\ref{FLSD3}). \rafaelC{Note that Theorem~\ref{theo:matrixTU} implies that we can concentrate only on integer solutions for (\ref{FLSDobj})-(\ref{FLSD3}).}

\begin{lemma}\label{lemma:singleitem}
Nondominated inequalities are only related to a single item $i\in I$.
\end{lemma}
\begin{proof}
Constraints (\ref{FLSD1})-(\ref{FLSD3}) do not relate variables connected to different items. This implies that (\ref{FLSDobj})-(\ref{FLSD3}) can be solved separately for each item. Thus, any inequality (\ref{ineq:projected}) which contains more than one item can be obtained as a linear combination of the constraints related to each item separately. 
\end{proof}

\begin{lemma}\label{lemma:thetaorgamma}
In a nondominated inequality, whenever $\phi^i_k = -1 $, for each period $t\leq k$ \marcioC{either (a) $\theta^{ij}_t = 1$ and $\gamma^{ij}_{tk} = 0$ or (b) $\theta^{ij}_t = 0$ and $\gamma^{ij}_{tk} = 1$.} 
\end{lemma}
\begin{proof}
Note that both $\theta^{ij}_t = 1$ and $\gamma^{ij}_{tk} = 1$ have nonnegative coefficients in the objective function (\ref{FLSDobj}). With $\phi^i_k = -1$, constraints (\ref{FLSD1}) require that $\theta^{ij}_t + \gamma^{ij}_{tk} \geq 1$ for every $t\leq k$. Whenever $\theta^{ij}_t = 1$,  $\gamma^{ij}_{tk}$ can be set to zero. \rafaelC{On the other hand}, note that whenever $\gamma^{ij}_{tk}=1$, $\gamma^{ij}_{tk'}=1$ for every $k'\geq k$ for which $\phi^i_{k'} = -1$\marcioC{, since $\theta^{ij}_t = 0$}.
\end{proof}

\begin{lemma}\label{lemma:onetol}
For a given item $i\in I$, if there is a most violated inequality (\ref{ineq:projected}) in which $\phi^i_k = -1 $ for a given $k>1$, then there is a most violated inequality in which $\phi^i_{k'} = -1 $ for every $k' < k$.
\end{lemma}
\begin{proof}

Assume there is a most violated inequality obtained as (\ref{FLSDobj}) represented by a solution $(\hat{\phi},\hat{\gamma},\hat{\theta})$ in which $\hat{\phi}^i_{k} = -1$.
We want to show that we can set $\hat{\phi}^i_{k'} = -1$ and obtain another most violated inequality. 
If we set $\hat{\phi}^i_{k'} = -1$, observe that constraints (\ref{FLSD1}) are already satisfied for every $t\leq k'$ such that $\hat{\theta}^{ij}_t = 1$.
Now consider the periods $t\leq k'$ such that $\hat{\theta}^{ij}_t = 0$ and note that $\hat{\gamma}^{ij}_{tk} = 1$. Let $T'$ be formed by all these periods. 
Observe that the summation $- d^{i}_k + \sum_{t\in T'} (\hat{\gamma}^{ij}_{tk} \hat{y}^j_t) d^i_k $ is less than or equal to zero as the inequality is a most violated one. Thus $- d^{i}_{k'} + \sum_{t\in T'} (\hat{\gamma}^{ij}_{tk'} \hat{y}^j_t) d^i_{k'}$ is also nonnegative. Thus, setting $\hat{\phi}^i_{k'}=-1$ and also $\hat{\gamma}^{ij}_{tk'} = 1$ for every $t \in T'$ leads to an inequality which is at least as violated as the original one. As this is true for any $k'<k$, the result holds.
\end{proof}

\begin{theorem}\label{theo:obtainlsj}
Every $(l,S_j)$-inequality can be obtained as (\ref{ineq:projected}).
\end{theorem}
\begin{proof}
Consider an $(l,S_j)$-inequality obtained for a given $i\in I$ and $l \in T$, with sets $S_j$ for each $j\in J$. This inequality can be obtained as (\ref{ineq:projected}) by considering as only nonzero values: 
\begin{itemize}
\item $\phi^i_k = -1$ for every $k \in L\rafaelC{=\{1,\ldots,l\}}$; 
\item $\theta^{ij}_t \marcioC{= 1}$ for every $j\in J$ and $t\in L\setminus S_j$;
\item $\gamma^{ij}_{tk} \marcioC{= 1}$ for every $j \in J$, $t\in S_j$ and $k \in L$ with $t \geq k$,
\end{itemize}
\rafaelC{
which leads to
\begin{equation*}
     \sum_{j=1}^{NJ} \sum_{t \in L\setminus S_j} \theta^{ij}_t {x}^{ij}_{t} + \sum_{j=1}^{NJ} \sum_{t \in S_j} \sum_{k=t}^{NT} \gamma^{ij}_{tk} d^{i}_k {y}^j_{t} + \sum_{t=1}^{l} \phi^i_t d^i_t \geq 0,
\end{equation*}
and, consequently, to
\begin{equation*}
     \sum_{j=1}^{NJ} \left( \sum_{t \in L\setminus S_j} {x}^{ij}_{t} + \sum_{t \in S_j} \sum_{k=t}^{NT} d^{i}_k {y}^j_{t} \right) \geq d^i_{1l},
\end{equation*}
which is equivalent to \eqref{ineq:lsj}.
}

\end{proof}

\begin{theorem}\label{theo:relation}
For every most violated inequality obtained as (\ref{FLSDobj}), there is a corresponding most violated $(l,S_j)$-inequality.
\end{theorem}
\begin{proof}
The result follows from Lemmas~\ref{lemma:singleitem},~\ref{lemma:thetaorgamma} and~\ref{lemma:onetol}. 
\end{proof}

\begin{corollary}\label{corollary:equalbounds}
Let $\underline{z}_{STD+}$ be the value of the linear relaxation of (\ref{STDobj})-(\ref{STD4}) with the addition of the inequalities (\ref{ineq:lsj}), and $\underline{z}_{FL}$ be the value of the linear relaxation of (\ref{FLobj})-(\ref{FL4}), then $\underline{z}_{STD+} = \underline{z}_{FL}$.
\end{corollary}
Corollary~\ref{corollary:equalbounds} follows from Theorems~\ref{theo:obtainlsj}~and~\ref{theo:relation}. 

\section{A simple MIP heuristic}
\label{sec:mipheuristic}

In this section, we show how to use the facility location formulation, which often provides strong relaxations, in a heuristic way. Note that its $O(NI\times NJ \times NT^2)$ variables and constraints \rafaelC{turn} the formulation prohibitive for being used to deal with large instances.

Let $K_{MH}$ be a constant integer given as input to the MIP heuristic. The MIP heuristic only considers variables $X^{ij}_{tk}$ defined for periods $t\in T$ and $k \in T$, with $t\leq k$ and $k \leq t + K-1$, i.e., variables corresponding to an interval of size $K_{MH}$. The MIP heuristic thus consists \rafaelC{of} solving the formulation
\begin{align}   
z_{FL(K_{MH})} = & \min \ \sum_{i=1}^{NI} \sum_{j=1}^{NJ} \sum_{t=1}^{NT} \sum_{\substack{k=t \\ k \leq t + K_{MH}-1 }}^{NT} P_{ij} X^{ij}_{tk}  + \sum_{j=1}^{NJ} \sum_{t=1}^{NT} O_j y^j_t +  \sum_{i=1}^{NI} \sum_{t=1}^{NT} H_i (\sum_{u=1}^t \sum_{\substack{k=t+1 \\ k \leq u+K_{MH} -1 }}^{NT} X^{ij}_{uk})  \label{HFLobj} \\
& \sum_{j=1}^{NJ}\sum_{\substack{t=1 \\ k \leq t + K_{MH}-1}}^{t}X^{ij}_{tk} = d^i_k, \qquad \textrm{for } i \in I, \ k \in T, \label{HFL1}\\
& X^{ij}_{tk}  \leq d^i_k y^j_t, \qquad \textrm{for }  i \in I,\ j\in J,\ t \in T,\ k \in \{ t,\ldots,\min \{t+K_{MH}-1,NT \} \}, \label{HFL2}\\
& y^j_{t} \in \{0,1\}, \qquad \textrm{for } j \in J, \ t \in T, \label{HFL3}\\
& X^{ij}_{tk} \geq 0, \qquad \textrm{for } i \in I,\ j \in J, \ t \in T, \ k\in \{t,\ldots,\min \{t+K_{MH}-1,NT \}\}.  \label{HFL4}
\end{align}
Note that the objective function and all \rafaelC{the} constraints are similar to those of the facility location formulation (\ref{FLobj})-(\ref{FL4}), differing only by the fact that solely a subset of the variables are considered.
This formulation has $O(NI\times NJ \times NT \times K_{MH})$ variables and constraints. \rafaelC{It is worth mentioning that the achieved reduction depends greatly on the size of the planning horizon, and it can get close to one order of magnitude as $K_{MH}$ gets smaller.}

\section{Computational experiments}
\label{sec:computationalexperiments}

This section summarizes the computational experiments conducted to assess the performance of the proposed approaches.
All computational experiments were carried out on a machine running under Ubuntu GNU/Linux, with an Intel(R) Core(TM) i7-8700 CPU @ 3.20GHz processor and 16Gb of RAM. The algorithms were coded in Julia v1.2.0, using JuMP v0.18.6. The formulations were solved using Gurobi 9.0.1 with the standard configurations, except the relative optimality tolerance gap which was set to $10^{-6}$. 
Subsection~\ref{sec:instances} describes the benchmark instances. Subsection~\ref{sec:testedapproaches} details the tested approaches \rafaelC{and parameter settings. Subsection~\ref{sec:computationalpreprocessing} assesses the effectiveness of the preprocessing approach. Subsection~\ref{sec:computationalexact} compares the exact mixed integer programming (MIP) formulations. 
Subsection~\ref{sec:computationalheuristic} displays the results for the MIP heuristic.
}

\subsection{Benchmark instances}
\label{sec:instances}

The computational experiments were performed using the original benchmark set of instances proposed by~\citeA{BasLeu05}, where more details can be obtained. Instances are assembled into instance groups, which are identified as $(NJ,NI,NT)$. Each instance group $(NJ,NI,NT)$ is composed of 15 randomly created instances with $NJ$ suppliers, $NI$ items, and $NT$ periods. All the data were generated using uniform distributions. The transaction costs lie in [1000,2000], the unitary purchase prices lie in [20,50], the holding costs lie in [1,5], and the demands lie in [1,200]. \rafaelC{The benchmark set contains ten instance groups, which are summarized in Table~\ref{tab:isnctances}, giving a total of 150 instances.}

\begin{table}[H]
\small
\color{black}
    \centering
    \begin{tabular}{c}
         \hline Instance groups  \\
         \hline $(3,3,10)$; $(3,3,15)$; $(4,4,10)$; $(4,4,15)$; $(5,5,20)$; $(10,10,50)$; \\
           $(15,15,100)$; $(20,20,100)$; $(20,20,200)$; $(50,50,200)$.\\
         \hline 
    \end{tabular}
    \caption{Dimensions of the instance groups.}
    \label{tab:isnctances}
\end{table}

\rafaelC{Furthermore, in order to analyze the performance of the newly proposed approaches under different cost configurations, we generated two new benchmark sets. The first one has all the data randomly determined similarly to the original instances, with exception of the transaction costs which lie in [10000,12000]. For the second new benchmark set, on the other hand, the transaction costs lie in [15000,17000] while the holding costs lie in [10,20]. Each of these new benchmark sets is composed of ten instance groups, as described in Table~\ref{tab:isnctances}, with 15 instances each. They are denoted as instances N1 and N2, correspondingly.
Observe that, when compared to the original instances, instances N1 have higher transaction costs, while instances N2 have higher transaction costs as well as increased holding costs.
}

\subsection{Tested approaches and parameter settings}
\label{sec:testedapproaches}

The following approaches were considered in the computational experiments:
\begin{enumerate}[(a)]
    \item STD: the standard formulation \eqref{STDobj}-\eqref{STD4};
    \item FL: the facility location formulation \eqref{FLobj}-\eqref{FL4};
    \rafaelC{\item PFL: the facility location formulation \eqref{FLobj}-\eqref{FL4} preprocessed using the results of section~\ref{sec:preprocess};}
    \item \rafaelC{BC: a branch-and-cut using the $(l,S_j)$-inequalities \eqref{ineq:lsj} based on STD;}
    \item MH: the MIP heuristic presented in Section \ref{sec:mipheuristic}.
\end{enumerate}

\subsubsection{\rafaelC{Implementation details and parameter settings}}

All tests for STD, FL, \rafaelC{PFL and BC} were carried out with a time limit of one hour (3600s). \rafaelC{A time limit of ten minutes (600s) was imposed for each execution of MH, which represents one-third of that allowed by \citeA{CarGonTre15}.
The choices of the parameters were defined based on preliminary experiments which took into consideration around 10\% of the original instances, randomly chosen, with varying sizes.}

\rafaelC{
The cutting planes for BC were implemented as solver callbacks. 
In each round of cut, the separation of the inequalities is performed using inspection for each item $i\in I$ and each selected interval $[k,l]$, with $1\leq k \leq l \leq NT$. An interval $[k,l]$ is defined in a way that every period in $[1,k-1]$ is forced to be in $L\setminus S_j$ for every $j\in J$ and a most violated inequality is determined with the appropriate choices of $S_j$ for each of the suppliers $j\in J$ considering the periods in $[k,l]$. We also tested a separation procedure using the more efficient algorithm described in Section~\ref{sec:separation}, which finds a most violated inequality for each item $i\in I$. We observed, in the preliminary experiments, that even though each round of cut could be performed in less computational time, the number of rounds to achieve good lower bounds became larger, implying larger overall computational times. For this reason, we did not use this separation procedure in the complete computational experiments.
}

\rafaelC{
Separation for violated cuts is only performed at the root node. The maximum number of rounds of cuts to be carried out by the solver was set to ten (the values 10, 20, 30, and 50 were tested).
We limited the size of intervals of the $(l,S_j)$-inequalities to be separated to five periods for instances with at most 50 periods and to two periods for those with at least 100 periods (the values 2, 5, 10, and 15 were tested).
}

\rafaelC{The values in $\{2,5,10\}$ were tested for the sizes of the intervals in the MIP heuristic, i.e., $K_{MH}$. The results for all these three configurations are reported in Section~\ref{sec:computationalheuristic}.}

\subsection{\rafaelC{Effectiveness of the preprocessing approach}}
\label{sec:computationalpreprocessing}

\rafaelC{
The results assessing the effectiveness of the preprocessing approach presented in Section~\ref{sec:preprocess} are summarized in Tables~\ref{tab:FLxPFLtype1}-\ref{tab:FLxPFLtype14}.
The values in each line represent average values over the corresponding instance group. The first column identifies the instance group. In the following columns, for FL and PFL, the tables present the average upper bound (ub), the average time in seconds to prove optimality for the instances that could be solved (time), the number of instances solved to optimality (\#opt), and the average open gap for those instances not solved to optimality (gap), which is determined for each instance as $100 \times \frac{ub - lb}{ub}$, where $lb$ is the lower bound achieved at the end of the execution. Besides, the last column (red) indicates the average reduction (in \%) of the $X$ variables eliminated by PFL using preprocessing.
The value '--' in the column time means that no instance in that group was solved to optimality, while its presence in the column gap indicates that all the instances in the group were optimally solved. Besides, the value 'n/a' represents the fact that executions were halted by the computer due to memory limitation.
}

\rafaelC{
Table~\ref{tab:FLxPFLtype1} shows the results for the original instances. The table indicates that considerable gains were achieved by PFL when compared to FL, especially for the larger instances.
Column red shows that substantial reductions were achieved by the preprocessing, starting from ~30\% for the smallest instance groups and reaching nearly 90\% for the instances with 100 periods.
Both approaches could solve to optimality all instances with at most 50 periods within a few seconds on average.
PFL shows a larger number of instances with 100 periods solved to optimality and lower average optimality gaps for the unsolved instances. 
We remark that the sizes of the formulations were prohibitive to be executed in the available computational resources for the larger instances with 200 periods, even with preprocessing.
}

\rafaelC{
Table~\ref{tab:FLxPFLtype2} displays the results for the new instances N1. It can be observed that both FL and PFL encountered more difficulties with these instances than with the original ones. It is noteworthy that due to the highest transaction costs, the achieved reductions were more modest, especially for the smaller instances with at most 20 periods, but these reductions still achieved around 50\% as the sizes of the instances increased.
None of the instances with at least 100 periods could be solved to optimality, and it can be observed that PFL achieved lower gaps for these instances.
Similar to what happened for the original instances, the formulations were prohibitively large to be executed in the available computational resources for the instances with 200 periods.
}

\rafaelC{
Table~\ref{tab:FLxPFLtype14} presents the results for the new instances N2. It can be noted that both approaches encountered more difficulties with these instances than with the original ones, but they were still more tractable than instances N1.
It can be observed that the reductions achieved by PFL were more modest when compared to those obtained for the original instances, especially for the smaller ones, but still more considerable than those for instances N1. Note, however, that the achieved reductions grew much larger as the sizes of the instances increased.
Considering the larger instance groups for which both FL and PFL solved the same amount of instances to optimality, (10,10,50) and (20,20,100), the reduction in the average times is remarkable.
Again, the formulations were prohibitively large to be solved using the available computational resources for the instances with 200 periods.
}


\begin{table}[H]
\color{black}
\scriptsize
    \centering
\begin{tabular}{c| c c c c| c c c c c }
\hline & \multicolumn{4}{c|}{FL}  & \multicolumn{5}{c}{PFL}  \\
Inst group & ub & time & \#opt &	gap & ub & time & \#opt &	gap & red \\	
\hline (3,3,10)	& 101940 & $<$0.1 &	15 & -- &	101940 &	$<$0.1 &	15 &	-- &	29.4 \\	
(3,3,15) &	147163 &	$<$0.1 &	15 &	-- &	147163 &	$<$0.1	 & 15 &	-- &	38.6 \\
(4,4,10) &	124526 &	$<$0.1 &	15 &	-- &	124526 &	$<$0.1 &	15 &	-- &	29.2 \\
(4,4,15) &	185681 &	$<$0.1 &	15 &	-- &	185681 &	$<$0.1 &	15 &	-- &	44.5 \\
(5,5,20) &	300866 &	0.1 &	15 &	-- &	300866 &	$<$0.1 &	15 &	-- &	53.5 \\
(10,10,50) &	1336982 &	8.0 &	15 &	-- &	1336982 &	1.6 &	15 &	-- &	77.5 \\
(15,15,100) &	3800975 &	541.8 &	8 &	0.07 &	3800975 &	289.6 &	12 &	0.06 &	88.9 \\	
(20,20,100) &	4950030 &	891.1 &	1 &	0.10 &	4949376 &	1383.1 &	7 &	0.07 &	88.3 \\	
(20,20,200) & n/a & n/a & n/a & n/a & n/a & n/a & n/a & n/a & n/a \\				
(50,50,200) & n/a & n/a & n/a & n/a & n/a & n/a & n/a & n/a & n/a \\				\hline
\end{tabular}
    \caption{Comparison between FL and PFL for the original instances.}
    \label{tab:FLxPFLtype1}
\end{table}

\begin{table}[H]
\color{black}
\scriptsize
    \centering
\begin{tabular}{c| c c c c| c c c c c }
\hline & \multicolumn{4}{c|}{FL}  & \multicolumn{5}{c}{PFL}  \\
Inst group & ub & time & \#opt &	gap & ub & time & \#opt &	gap & red \\	
\hline (3,3,10) &	126469 &	$<$0.1 &	15 &	-- &	126469 &	$<$0.1 &	15 &	-- &	0.0 \\
(3,3,15) &	199917 &	$<$0.1 &	15 &	-- &	199917 &	$<$0.1 &	15 &	-- &	0.1 \\
(4,4,10) &	161299 &	$<$0.1 &	15 &	-- &	161299 &	$<$0.1 &	15 &	-- &	0.0 \\
(4,4,15) &	245587 &	$<$0.1 &	15 &	-- &	245587 &	$<$0.1 &	15 &	-- &	0.4 \\
(5,5,20) &	387859 &	0.1 &	15 &	-- &	387859 &	0.1 &	15 &	-- &	1.7 \\
(10,10,50) &	1706395 &	301.0 &	13 &	0.19 &	1706395 &	402.6 &	14 &	0.22 &	22.3 \\
(15,15,100) &	4870354 &	-- &	0 &	0.82 &	4868906 &	-- &	0 &	0.78 &	49.3 \\
(20,20,100) &	6217228 &	-- &	0 &	0.86 &	6215298 &	-- &	0 &	0.82 &	50.2 \\
(20,20,200)  & n/a & n/a & n/a & n/a & n/a & n/a & n/a & n/a & n/a \\					
(50,50,200)  & n/a & n/a & n/a & n/a & n/a & n/a & n/a & n/a & n/a \\					\hline
\end{tabular}
    \caption{Comparison between FL and PFL for the new instances N1.}
    \label{tab:FLxPFLtype2}
\end{table}

\begin{table}[H]
\color{black}
\scriptsize
    \centering
\begin{tabular}{c| c c c c| c c c c c }
\hline & \multicolumn{4}{c|}{FL}  & \multicolumn{5}{c}{PFL}  \\
Inst group & ub & time & \#opt &	gap & ub & time & \#opt &	gap & red \\	
\hline (3,3,10) &	188177.9 &	$<$0.1 &	15 &	-- &	188177.9 &	$<$0.1 &	15 &	-- &	6.1 \\
(3,3,15) &	284028.9 &	$<$0.1 &	15 &	-- &	284028.9 &	$<$0.1 &	15 &	-- &	20.6 \\
(4,4,10) &	229132.5 &	$<$0.1 &	15 &	-- &	229132.5 &	$<$0.1 &	15 &	-- &	6.3 \\
(4,4,15) &	349938.3 &	$<$0.1 &	15 &	-- &	349938.3 &	$<$0.1 &	15 &	-- &	21.2 \\
(5,5,20) &	535305.7 &	0.1 &	15 &	-- &	535305.7 &	0.1 &	15 &	-- &	34.5 \\
(10,10,50) &	2270084.6 &	435.6 &	14 &	0.28 &	2270084.6 &	129.7 &	14 &	0.08 &	68.4 \\
(15,15,100) &	6207751.7 &	616.4 &	10 &	0.12 &	6207751.7 &	648.5 &	13 &	0.08 &	83.1 \\
(20,20,100) &	7620294.5 &	409.8 &	8 &	0.24 &	7620909.7 &	51.9 &	8 &	0.13 &	83.3 \\
(20,20,200)  & n/a & n/a & n/a & n/a & n/a & n/a & n/a & n/a & n/a \\					
(50,50,200)	 & n/a & n/a & n/a & n/a & n/a & n/a & n/a & n/a & n/a \\					\hline
\end{tabular}
    \caption{Comparison between FL and PFL for the new instances N2.}
    \label{tab:FLxPFLtype14}
\end{table}

\subsection{\rafaelC{Performance of the exact approaches}}
\label{sec:computationalexact}

\rafaelC{
Tables~\ref{tab:exacttype1}-\ref{tab:exacttype14} summarize the results comparing the exact approaches STD, PFL and BC. Note that, as it was evidenced in the previous subsection that PFL clearly outperforms FL, we do not present the results for the latter in this subsection.
In each of these tables, the first column gives the instance group. In the following columns, for STD, PFL, and BC, the tables present the average solver linear relaxation bound (glp), which is given for each instance by the lower bound reported by Gurobi at the end of its execution of the root node, the average upper bound (ub), the average time in seconds to prove optimality for the instances that could be solved (time), the number of instances solved to optimality (\#opt), and the average open gap for those not solved to optimality (gap). Additionally, for PFL, the table also gives the linear relaxation bound (lp). We report such value because it represents the theoretical linear relaxation bound of FL, PFL and that obtained when using STD and separating all the $(l,S_j)$-inequalities. Moreover, we can notice that the solution of the root node for BC is closer to such value than the solution of STD which supports our claim that the inequalities play an important role in tightening the optimality gap and strengthening the formulation.
}

\rafaelC{
Table~\ref{tab:exacttype1} shows the results for the original instances. It can be noticed that PFL outperformed STD and BC for all instance groups with up to 100 periods in terms of average times and number of instances solved to optimality, but it was far from being a good alternative for the instances with 200 periods, as it was already observed in Section~\ref{sec:computationalpreprocessing}. Regarding such instances, BC outperformed STD for the set of instances (20,20,200) but it is outperformed by the former for the set of instances (50,50,200). Nevertheless, considering all the computational experiments BC outperforms STD as it obtains smaller gaps, lower average times, and tighter lower bounds at the end of the processing of the root node, as presented by column glp. Although we do not explicitly show the number of enumerated nodes due to the already large number of columns in the tables, we could observe that both FL and BC outperform STD when we consider the number of enumerated nodes. As a final remark, notice that the bound at the end of the root node given by PFL is close to the pure linear relaxation bound, which is already better than the ones provided by STD and BC. 
}

\rafaelC{
Tables~\ref{tab:exacttype2}~and~\ref{tab:exacttype14} display the results for the new instances N1 and N2, respectively. It can be observed that the remarks made in the previous paragraph regarding the comparison between the performance of the three methods also hold, namely, that PFL outperforms BC and STD for the instance groups which could be processed without memory issues. Again, BC outperforms STD for all instance groups but the largest ones, i.e. (50,50,200). Also, one can argue that the instances N2 are probably more difficult than the original instances and that the instances N1 are possibly even harder. PFL was able to solve all the original instances up to size (20,20,100) but, for the set N2, it was only able to solve some instances with this size. For the set N1, on the other hand, it was only able to solve instances with sizes up to (15,15,100).  STD and BC were only able to solve instances with sizes up to (15,15,100) for N2 and (5,5,20) for N1. Furthermore, the gaps obtained by PFL for the unsolved instances are also larger for N2 than for the original set, and they are even larger for N1, which was not exactly the case for STD and BC.
}

\subsection{\rafaelC{Results for the MIP heuristic}}
\label{sec:computationalheuristic}

\rafaelC{
The summary of the results for the MIP heuristic is presented in Tables~\ref{tab:heurtype1}-\ref{tab:heurtype14}.
The first column represents the instance group. Next, for each of the three tested configurations of the MIP heuristic, denoted as MH-$K_{MH}$, the table shows the average solution value (ub), the average running time in seconds (time), and the percentual average optimality gap (gap$_{MH}$), obtained for each instance as $100 \times \frac{ub - bestlb}{ub}$, where $bestlb$ represents the best lower bound amongst those obtained with STD, PFL, and BC.
After that, column BestMIP presents the best average value considering the exact approaches tested in Section~\ref{sec:computationalexact}. 
For the original instances, Table~\ref{tab:heurtype1} also provides the best average upper bounds reported in \citeA{BasLeu05} and
\citeA{CarGonTre15}, identified by BL05 and CGT15, correspondingly. For CGT15, the presented value takes into consideration for each instance the best amongst the two variants of the heuristic described in their work. The best average heuristic solution values are shown in bold.} \marcioC{The careful reader might notice that some of the average times reported are bigger than the 600 seconds time limit, indicating that some of the heuristics overpass the limit time. Such observation is true and happens because the solver may take some extra time to finish the solution of the current node.
}

\begin{landscape}

\begin{table}[H]
\color{black}
\scriptsize
    \centering
\begin{tabular}{c| c c c c c| c c c c c c| c c c c c } \hline
&  \multicolumn{5}{c|}{STD} & \multicolumn{6}{c|}{PFL}  & \multicolumn{5}{c}{BC} \\
Inst group	&glp	&ub	&time	&\#opt	&gap	&lp	&glp	&ub	&time	&\#opt	&gap	&glp	&ub	&time	&\#opt	&gap\\ \hline
(3,3,10) &101509	&101940	&0.1	&15	& --  	&101920	&101935	&101940	& $<$0.1	&15	& --  	&101885	&101940	&0.6	&15	& -- \\
(3,3,15) &146366	&147163	&0.1	&15	&  -- 	&147151	&147156	&147163	& $<$0.1	&15	& --  	&147118	&147163	&0.6	&15	&  -- \\
(4,4,10) &124263	&124526	&0.1	&15	& --  	&124518	&124526	&124526	& $<$0.1	&15	& --  	&124509	&124526	&0.6	&15	&  -- \\
(4,4,15) &184696	&185681	&0.2	&15	& --  	&185674	&185681	&185681	& $<$0.1	&15	& --  	&185615	&185681	&0.7	&15	& --  \\
(5,5,20) &298964	&300866	&1.2	&15	& --  	&300793	&300839	&300866	& $<$0.1	&15	& --  	&300498	&300866	&1.1	&15	& --  \\
(10,10,50)	&1318889	&1336982	&540.0	&14	&0.02	&1336539	&1336606	&1336982	&1.6	&15	&  --	&1332484	&1336982	&205.5	&14	&0.04\\
(15,15,100)	&3693753	&3802248	&1740.6	&4	&0.15	&3798505	&3798669	&3800975	&289.6	&12	&0.06	&3736159	&3801012	&1631.1	&4	&0.07\\
(20,20,100)	&4782245	&5030121	&  -- 	&0	&4.53	&4944778	&4944913	&4949376	&1383.1	&7	&0.07	&4836356	&4949921	& --  	&0	&0.12\\
(20,20,200)	&9244459	&10578643	&  -- 	&0	&12.60	& n/a	& n/a	& n/a	& n/a	& n/a	& n/a	& 9563041	&9974134	& --  	&0	&4.41\\
(50,50,200)	&21090139	&25140386	&  -- 	&0	&16.18	& n/a	& n/a	& n/a	& n/a	& n/a	& n/a	& 20797964	&26210847	& --  	&0	&20.61\\ 
\hline
\end{tabular}
    \caption{Comparison between STD, PFL and BC for the original instances.}
    \label{tab:exacttype1}
\end{table}

\begin{table}[h]
\color{black}
\scriptsize
    \centering
\begin{tabular}{c| c c c c c| c c c c c c| c c c c c } \hline
&  \multicolumn{5}{c|}{STD} & \multicolumn{6}{c|}{PFL}  & \multicolumn{5}{c}{BC} \\
Inst group	&glp	&ub	&time	&\#opt	&gap	&lp	&glp	&ub	&time	&\#opt	&gap	&glp	&ub	&time	&\#opt	&gap\\ \hline
(3,3,10) &125604	&126469	& $<$0.1	&15	&  --	&126444	&126453	&126469	& $<$0.1	&15	&  --	&126361	&126469	&0.6	&15	&  \\
(3,3,15) &195138	&199917	&0.2	&15	&  --	&199586	&199720	&199917	& $<$0.1	&15	&  --	&199168	&199917	&0.7	&15	&  \\
(4,4,10) &159774	&161299	&0.1	&15	&  --	&161158	&161205	&161299	& $<$0.1	&15	&  --	&161112	&161299	&0.6	&15	&  \\
(4,4,15) &239576	&245587	&0.5	&15	&  --	&245237	&245375	&245587	& $<$0.1	&15	&  --	&244681	&245587	&0.8	&15	&  \\
(5,5,20) &372779	&387859	&5.3	&15	&  --	&387133	&387317	&387859	&0.1	&15	&0.00	&385056	&387859	&2.7	&15	&  \\
(10,10,50)	&1573909	&1707243	&  --	&0	&2.86	&1699417	&1699517	&1706395	&402.6	&14	&0.22	&1634584	&1706790	&  --	&0	&2.13\\
(15,15,100)	&4269343	&4990980	&  --	&0	&15.02	&4827029	&4827116	&4868906	&  --	&0	&0.78	&4252329	&4919194	&  --	&0	&7.86\\
(20,20,100)	&5463144	&6528403	&  --	&0	&16.81	&6160809	&6160863	&6215298	&  --	&0	&0.82	&5437357	&6414232	&  --	&0	&15.28\\
(20,20,200)	&10015474	&13630978	&  --	&0	&26.70	& n/a	& n/a	& n/a	& n/a	& n/a	& n/a	&10470462	&13396444	&  --	&0	&22.18\\
(50,50,200)	&21831899	&32707523	&  --	&0	&33.19	& n/a	& n/a	& n/a	& n/a	& n/a	& n/a	&21668579	&33205632	&  --	&0	&34.71\\
\hline
\end{tabular}
    \caption{Comparison between STD, PFL and BC for the new instances N1.}
    \label{tab:exacttype2}
\end{table}

\begin{table}[h]
\color{black}
\scriptsize
    \centering
\begin{tabular}{c| c c c c c| c c c c c c| c c c c c } \hline
&  \multicolumn{5}{c|}{STD} & \multicolumn{6}{c|}{PFL}  & \multicolumn{5}{c}{BC} \\
Inst group	&glp	&ub	&time	&\#opt	&gap	&lp	&glp	&ub	&time	&\#opt	&gap	&glp	&ub	&time	&\#opt	&gap\\ \hline
(3,3,10) &185527	&188178	&0.1	&15	& --  	&188139	&188165	&188178	& $<$0.1	&15	&  --	&187657	&188178	&0.6	&15	&  \\
(3,3,15) &275458	&284029	&0.3	&15	& --  	&284029	&284029	&284029	& $<$0.1	&15	&  --	&282053	&284029	&0.7	&15	&  \\
(4,4,10) &224957	&229133	&0.2	&15	& --  	&229049	&229071	&229133	& $<$0.1	&15	&  --	&227504	&229133	&0.7	&15	&  \\
(4,4,15) &337840	&349938	&1.1	&15	& --  	&349314	&349609	&349938	& $<$0.1	&15	&  --	&345123	&349938	&1.3	&15	&  \\
(5,5,20) &510609	&535306	&13.0	&15	& --  	&534088	&534304	&535306	&0.1	&15	&  --	&526143	&535306	&14.0	&15	&  \\
(10,10,50)	&2041641	&2270182	&601.0	&2	&1.62	&2256790	&2257975	&2270085	&129.7	&14	&0.08	&2134482	&2270110	&1403.8	&3	&1.38\\
(15,15,100)	&5288732	&6344819	&  --	&0	&17.17	&6196495	&6197144	&6207752	&648.5	&13	&0.08	&5203324	&6233441	&  -- 	&0	&1.84\\
(20,20,100)	&6544370	&8167961	&  --	&0	&20.18	&7605744	&7606692	&7620910	&51.9	&8	&0.13	&6393756	&7976353	&  -- 	&0	&19.28\\
(20,20,200)	&11465361	&17640815	&  --	&0	&35.11	& n/a	& n/a	& n/a	& n/a	& n/a	& n/a	&11907617	&17281521	&  -- 	&0	&31.51\\
(50,50,200)	&22433576	&39154536	&  --	&0	&42.65	& n/a	& n/a	& n/a	& n/a	& n/a	& n/a	&22087540	&40837776	&  -- 	&0	&45.83\\
\hline
\end{tabular}
    \caption{Comparison between STD, PFL and BC for the new instances N2.}
    \label{tab:exacttype14}
\end{table}

\end{landscape}

\begin{landscape}

\begin{table}[H]
\color{black}
\scriptsize
    \centering
\begin{tabular}{c| c c c| c c c| c c c | c | c | c } \hline
	& \multicolumn{3}{c|}{MH-2}	& \multicolumn{3}{c|}{MH-5}	& \multicolumn{3}{c|}{MH-10}	& BestMIP & BL05	& CGT15	\\
Inst group	&ub	&time	&gap$_{MH}$	&ub	&time	&gap$_{MH}$	&ub	&time	&gap$_{MH}$	&ub	&ub	&ub\\ \hline
(3,3,10) &102237	&$<$0.1	&0.29	&\textbf{101940}	&$<$0.1	&0.00	& \textbf{101940}	&$<$0.1	&0.00	&101940 &102584	&101954	\\
(3,3,15) &147912	&$<$0.1	&0.51	&\textbf{147163}	&$<$0.1	&0.00	&\textbf{147163}	&$<$0.1	&0.00	&147163 &147887	&\textbf{147163}	\\
(4,4,10) &125122	&$<$0.1	&0.49	&\textbf{124526}	&$<$0.1	&0.00	&\textbf{124526}	&$<$0.1	&0.00	&124526 &126345	&\textbf{124526}	\\
(4,4,15) &186073	&$<$0.1	&0.23	&185688	&$<$0.1	&0.00	&\textbf{185681}	&$<$0.1	&0.00	&185681 &187320	&185699	\\
(5,5,20) &301666	&$<$0.1	&0.28	&300881	&$<$0.1	&0.01	&\textbf{300866}	&$<$0.1	&0.00	&300866 &303529	&300900	\\
(10,10,50)	&1338660	&0.9	&0.13	&1336991	&1.3	&0.00	&\textbf{1336982}	&1.5	&0.00	&1336982 &1357190	&1337662	\\
(15,15,100)	&3806436	&302.9	&0.16	&3801107	&266.1	&0.02	&\textbf{3800977}	&272.0	&0.01	&3800975 &3856800	&3810899	\\
(20,20,100)	&4954723	&517.1	&0.15	&4949464	&531.9	&0.04	&\textbf{4949401}	&535.9	&0.04	&4949376 &5048826	&4975149	\\
(20,20,200)	&9825028	&600.1	&2.97	&9818199	&602.0	&2.90	&\textbf{9818136}	&600.8	&2.90	&9974134 &10026074	&9914548	\\
(50,50,200)	&\textbf{23120817}	&601.8	&8.86	&28295969	&629.5	&18.36	&25002196	&616.3	&12.32	&25140386 &25373121	&23457449	\\ \hline
\end{tabular}
    \caption{Results obtained by MH for the original instances.}
    \label{tab:heurtype1}
\end{table}

\begin{table}[H]
\color{black}
\scriptsize
    \centering
\begin{tabular}{c| c c c| c c c| c c c | c } \hline
	& \multicolumn{3}{c|}{MH-2}	& \multicolumn{3}{c|}{MH-5}	& \multicolumn{3}{c|}{MH-10}	&BestMIP\\
Inst group	&ub	&time	&gap$_{MH}$	&ub	&time	&gap$_{MH}$	&ub	&time	&gap$_{MH}$	&ub\\ \hline
(3,3,10) &136452	&$<$0.1	&7.36	&126888	&$<$0.1	&0.34	&\textbf{126469}	&$<$0.1	&0.00	&126469\\
(3,3,15) &212629	&$<$0.1	&5.99	&201466	&$<$0.1	&0.74	&\textbf{200112}	&$<$0.1	&0.09	&199917\\
(4,4,10) &172487	&$<$0.1	&6.56	&162793	&$<$0.1	&0.93	&\textbf{161299}	&$<$0.1	&0.00	&161299\\
(4,4,15) &256086	&$<$0.1	&4.18	&247462	&$<$0.1	&0.80	&\textbf{245715}	&$<$0.1	&0.06	&245587\\
(5,5,20) &402094	&$<$0.1	&3.59	&389528	&0.1	&0.45	&\textbf{387930}	&0.1	&0.02	&387859\\
(10,10,50)	&1764190	&155.5	&3.30	&1712482	&142.1	&0.37	&\textbf{1706466}	&164.6	&0.02	&1706395\\
(15,15,100)	&5003845	&600.2	&3.45	&4881966	&600.1	&1.04	&\textbf{4871949}	&600.3	&0.84	&4868906\\
(20,20,100)	&6340071	&600.1	&2.77	&6230498	&600.0	&1.06	&\textbf{6216874}	&600.2	&0.84	&6215298\\
(20,20,200)	&12709547	&600.0	&17.99	&\textbf{12457212}	&602.1	&16.32	&12597927	&605.9	&17.23	&13396444\\
(50,50,200)	& \textbf{107583588}	&600.7	&79.72	& n/a	& n/a	& n/a	&118787855	&619.7	&81.63	&32707523\\
\hline
\end{tabular}
    \caption{Results obtained by MH for the new instances N1.}
    \label{tab:heurtype2}
\end{table}

\begin{table}[H]
\color{black}
\scriptsize
    \centering
\begin{tabular}{c| c c c| c c c| c c c | c } \hline
	& \multicolumn{3}{c|}{MH-2}	& \multicolumn{3}{c|}{MH-5}	& \multicolumn{3}{c|}{MH-10}	&BestMIP\\
Inst group	&ub	&time	&gap$_{MH}$	&ub	&time	&gap$_{MH}$	&ub	&time	&gap$_{MH}$	&ub\\ \hline
(3,3,10) &189082	&$<$0.1	&0.49	&\textbf{188178}	&$<$0.1	&0.00	&\textbf{188178}	&$<$0.1	&0.00	&188178\\
(3,3,15) &284554	&$<$0.1	&0.19	&\textbf{284029}	&$<$0.1	&0.00	&\textbf{284029}	&$<$0.1	&0.00	&284029\\
(4,4,10) &229154	&$<$0.1	&0.01	&\textbf{229133}	&$<$0.1	&0.00	&\textbf{229133}	&$<$0.1	&0.00	&229133\\
(4,4,15) &\textbf{349938}	&$<$0.1	&0.00	&\textbf{349938}	&$<$0.1	&0.00	&\textbf{349938}	&$<$0.1	&0.00	&349938\\
(5,5,20) &535611	&$<$0.1	&0.06	&\textbf{535306}	&0.1	&0.00	&\textbf{535306}	&0.1	&0.00	&535306\\
(10,10,50)	&2270086	&69.5	&0.01	&\textbf{2270085}	&109.7	&0.01	&\textbf{2270085}	&142.7	&0.01	&2270085\\
(15,15,100)	&\textbf{6207752}	&198.4	&0.01	&\textbf{6207752}	&241.4	&0.01	&6207832	&260.3	&0.01	&6207752\\
(20,20,100)	&\textbf{7620910}	&286.8	&0.07	&7620997	&294.9	&0.07	&7620945	&304.1	&0.07	&7620910\\
(20,20,200)	&\textbf{15332550}	&522.1	&22.82	&15333970	&533.6	&22.83	&15333724	&533.9	&22.83	&17281521\\
(50,50,200)	&\textbf{32834743}	&601.2	&31.67	&32837872	&608.4	&31.68	&175160673	&649.1	&84.30	&39154536\\
\hline
\end{tabular}
    \caption{Results obtained by MH for the new instances N2.}
    \label{tab:heurtype14}
\end{table}

\end{landscape}

\rafaelC{
Table~\ref{tab:heurtype1} shows the results for the original instances. MH-10 achieved the best results for all instance groups but the largest one. MH-5 and MH-10 achieved solutions within 0.1\% of optimality for all instance groups with at most 100 periods. For the largest instance group, (50,50,200), MH-2 achieved the best average solutions, indicating that, even with the heuristic reduction, the formulations of MH-5 and MH-10 already became too large to be reasonably tractable using the available computational resources.
It is noteworthy that, using our heuristics, it was possible to obtain average values which are at least as good as those of BL05 and CG15 for all instance groups, with strictly better values for eight out of the ten. Furthermore, heuristic solution values improving those using the exact MIP approaches within the time limit were obtained for the larger instance groups with 200 periods. Note that even with the solver limited to run for at most 600 seconds, a few reported average values for the larger instances with 200 periods using MH-5 and MH-10 are a little higher than this allowed time limit. The reason for that is related to difficulties of the solver in finishing its execution for certain instances, possibly encountered in the tasks of freeing the memory given the sizes of the formulations or finishing a step of the solution method used by the solver (commonly the barrier method).
}

\rafaelC{
Table~\ref{tab:heurtype2} presents the results for the new instances N1. For these instances, MH-10 achieved average gaps below 1.0\% for all instances with at most 100 periods. 
MH-5 and MH-2 obtained the best average solution values for the groups (20,20,200) and (50,50,200), respectively. It is noteworthy that MH-2 obtained much larger gaps for these instances when compared to its results for the original instances. The reason for this behavior is probably related to the fact that the larger transaction costs imply fewer periods with setups and thus larger intervals between orders for low-cost solutions. Note that the MIP heuristic did not perform very well for the instance group (50,50,200), as the average solution values are much higher than that of the best exact approach considering the time limit. These observations strengthen the argument that the most difficult tested instances are probably those in set N1. We remark that the values 'n/a' observed for MH-5 in the instance group (50,50,200) are related to difficulties of the solver in finishing its execution, probably due to memory issues or numerical difficulties in generating a basic feasible solution after solving the linear relaxation using the barrier method.
}

\rafaelC{
Table~\ref{tab:heurtype14} displays the results for the new instances N2. Variants MH-5 and MH-10 obtained average gaps below 0.1\% for all the instance groups with at most 100 periods. MH-5 obtained the best average solution values for seven out of the ten instance groups. For this instance set, MH-2 outperformed the other variants for the instance groups with at least 100 periods, as it obtained the best average values for all of them. We remark, though, that differently from what was observed for the original instances, there is still a reasonably high open gap for the instances with 200 periods. One possible reason is the fact that the setup costs are very high, and thus, certain setup decisions can strongly influence the costs of the solutions.
}

\section{Final conclusions}

In this paper, we considered the multi-item inventory lot-sizing problem with supplier selection. The complexity of the problem was an open question and thus we have shown that it is NP-hard. Moreover, we have proposed a facility location extended formulation together with a preprocessing scheme, valid inequalities in the original space of variables, and an easy to implement mixed integer programming (MIP) heuristic. \rafaelC{Additionally, we introduced two new benchmark sets of instances with different cost parameters to complement the original benchmark set in order to better assess the performance of the proposed approaches.}


\rafaelC{
Computational experiments have shown that the preprocessing scheme was able to reach a considerable reduction in the number of variables considered for optimization. This established the preprocessed facility location formulation as a very effective approach for optimally solving instances with up to 100 periods, as nearly all the instances in the original and N2 sets could be solved to optimality. 
Besides, the valid inequalities implemented in a branch-and-cut approach could \marcioC{successfully} improve the capacity of the solver to deal with all but the largest instance groups (50,50,200). 
Finally, the proposed MIP heuristic was able to encounter high-quality results, outperforming those obtained by a state-of-the-art approach.
The newly proposed benchmark instances have shown to be more challenging for the proposed approaches than the original set available in the literature, especially instances N1.
}

\rafaelC{A possible direction for future research is a polyhedral study of the multi-item inventory lot-sizing problem with supplier selection. Besides, we}
\marcioC{
remark that the problem treated in this paper is for a two-echelon supply chain composed of one buyer and multiple suppliers. Thus, it would be interesting to explore this problem in a multi-echelon supply chain. Moreover, in a real environment, one might be subject to certain types of constraints or situations not considered in this work. In many agricultural supply chains, for instance, one is subject to shortages of products due to natural incidents. In the service and supply industry, one might face backordering due to obstacles in the deliveries. Such distinct features are some research avenues that could be investigated in the future.}
\marcioC{Finally, we point out that all the models treated in this work are deterministic and it would be also interesting to model and study the impact of uncertain client demands, which are more realistic in some fields.}

\label{sec:conclusions}

\vspace{0.8cm}

{
\noindent \small 
\textbf{Acknowledgments:}
Work of Rafael A. Melo was supported by Universidade Federal da Bahia, the Brazilian Ministry of Science, Technology, Innovation and Communication (MCTIC); the State of Bahia Research Foundation (FAPESB); and the Brazilian National Council for Scientific and Technological Development (CNPq). \rafaelC{The authors are thankful to four anonymous reviewers whose insightful comments helped to significantly improve the quality of this paper.}
}

\bibliography{main}

\bibliographystyle{apacite}

\end{document}